\documentclass[a4paper,leqno]{amsart}

\usepackage{amsmath}
\usepackage{amsthm}
\usepackage{amssymb}
\usepackage{amsfonts}
\usepackage[applemac]{inputenc}
\usepackage{mathrsfs}
\usepackage{pdfsync}
\usepackage{color} 
\usepackage{mathtools}

\def\R{{\mathbb R}}% real numbers
\def\N{{\mathbb N}}% nonnegative integers
\def\le{\leqslant}% lessoreqal
\def\ge{\geqslant}%greaterorequal

\newcommand{\re}{\mathrm{Re}}
\newcommand{\im}{\mathrm{Im}}
\newcommand{\eps}{\varepsilon}

\theoremstyle{plain}
\newtheorem{theorem}{Theorem}[section]
\newtheorem{lemma}[theorem]{Lemma}
\newtheorem{corollary}[theorem]{Corollary}
\newtheorem{proposition}[theorem]{Proposition}
\newtheorem{hyp}{Assumption}

\theoremstyle{definition}

\newtheorem{remark}[theorem]{Remark}
\newtheorem*{remark*}{Remark}

\numberwithin{equation}{section}

\begin{document}

\title[(In-)Stability properties of rotating BEC]{Stability and instability properties of rotating Bose-Einstein condensates}

\author[J. Arbunich]{Jack Arbunich}

\author[I. Nenciu]{Irina Nenciu}

\author[C. Sparber]{Christof Sparber}

\address[J. Arbunich]
{Department of Mathematics, Statistics, and Computer Science, M/C 249, University of Illinois at Chicago, 851 S. Morgan Street, Chicago, IL 60607, USA}
\email{jarbun2@uic.edu}

\address[I.~Nenciu]
{Department of Mathematics, Statistics, and Computer Science, M/C 249, University of Illinois at Chicago, 851 S. Morgan Street, Chicago, IL 60607, USA  \textit{and} Institute of Mathematics ``Simion Stoilow''
     of the Romanian Academy\\ 21, Calea Grivi\c tei\\010702-Bucharest, Sector 1\\Romania}
\email{nenciu@uic.edu}

\address[C.~Sparber]
{Department of Mathematics, Statistics, and Computer Science, M/C 249, University of Illinois at Chicago, 851 S. Morgan Street, Chicago, IL 60607, USA}
\email{sparber@uic.edu}

\begin{abstract}
We consider the mean-field dynamics of Bose-Einstein condensates in rotating harmonic traps and establish several stability and instability properties 
for the corresponding solution. We particularly emphasize the difference between the situation in which the trap is symmetric with respect to the rotation axis 
and the one where this is not the case.
\end{abstract}

\date{\today}


\thanks{This publication is based on work supported by the NSF through grant nos. DMS-1348092
 and DMS-1150427. 
The authors want to thank Robert Seiringer and Michael Loss for 
inspiring discussions}
\maketitle

%%%%%%%%%%%%%%%%%%%%%%%%%%%%%%%%%%%%%%%%%%%%%%%%%%%%%%%%%%%%%%%%%%%%%%%%

\section{Introduction}\label{sec:intro}

In this note, we consider the dynamics of (harmonically) trapped {\it Bose-Einstein condensates} (BEC), subject to 
an external rotating force. Because of their ability to display quantum effects at the macroscopic scale, BEC 
have become an important subject of research, both experimentally and theoretically. In particular, the 
expression of quantum vortices in rapidly rotating BEC has been an ongoing topic of interest over the last few decades, see, e.g., \cite{Af, BWM, CaCl, Co, DaSt, Fe, SaUe} and the 
references therein.
It is well-known that in the mean-field regime, BEC can be accurately described by 
the celebrated {\it Gross-Pitaevskii equation} (GP) for $\psi$, the macroscopic wave function of the condensate, see \cite{LSY, RS1, RS2}. 
In dimensionless units, the GP equation with general nonlinearity reads
\begin{equation}\label{GP}
	i \partial_{t}\psi = - \frac{1}{2}\Delta \psi + V(x)\psi + a|\psi|^{2\sigma}\psi - (\Omega \cdot L) \psi\,, \quad \psi_{\vert{t=0}} = \psi_{0}(x)\,.
\end{equation}
Here, $a\in \R$, $\sigma>0$ and $(t,x)\in \R\times \R^3$ with $d=2$, or $3$, respectively. 
The former situation thereby corresponds to the case of an {\it effective} 
two-dimensional BEC, obtained via strong confining forces, see, e.g., \cite{MeSp} for more details. 
The external potential $V(x)\in \R$ is assumed to be {\it harmonic}, i.e., 
\begin{equation}\label{eq:pot}
V(x) = \frac{1}{2}\sum_{j=1}^{d}\omega_{j}^{2}x_{j}^{2},
\end{equation}
where the parameters $\omega_{j} \in \R\setminus \{0\}$ represent the respective trapping frequencies in each spatial direction. 
As we shall see, the smallest trapping frequency denoted by
\[
0<\omega \equiv  \min\limits_{j=1,\dots,d} \{\omega_{j}\},
\]
will play a particular role in our analysis. 

We further assume that the BEC is subject to a rotating force along a given {\it rotation axis} $\Omega\in \R^3$ and denote by 
\begin{equation*}
L = -i x \wedge \nabla,
\end{equation*}
the quantum mechanical angular momentum operator. Note that in dimension $d=2$, we always have
\begin{equation}\label{op2d}
\Omega \cdot L = -i|\Omega|\big(x_1 \partial_{x_2} - x_2\partial_{x_1}\big),
\end{equation}
corresponding to the case where $\Omega = (0,0,|\Omega|)\in \R^3$. 

The nonlinearity in \eqref{GP} describes the mean-field self-interaction of the condensate particles. The physically most relevant case 
is given by a cubic nonlinearity, i.e. $\sigma=1$, but for the sake of generality we shall in the following allow for more general $\sigma>0$. We shall also 
allow for both {\it attractive} $a<0$ and {\it repulsive} $a>0$ interactions, satisfying Assumption \ref{hyp1} below. 
Vortices are generally believed to be unstable in the former case (see, e.g., \cite{CaCl, CLS, SaUe}), 
while they are known to form stable lattice configurations in the latter \cite{Af, Co, Fe}. 

In this work, we shall not be interested in the dynamical features of individual vortices, but rather study 
bulk properties of the condensate, as described by \eqref{GP}. To this end, we recall that 
the natural energy space associated to \eqref{GP} is given by 
\[
\Sigma = \{ u \in H^1(\R^d) \, : \,  |x|u \in L^2(\R^d)\},
\]
equipped with the norm
\[
\|u\|_{\Sigma}^2=\|u\|_{L^2}^2+\|\nabla u\|_{L^2}^2+\||x|u\|_{L^2}^2\,.
\]
We also impose the following 
sub-criticality condition on the nonlinearity:
\begin{hyp} \label{hyp1}
One of the following holds:
\begin{itemize}
\item $a > 0$ (defocusing) and $0 < \sigma < \frac{2}{(d-2)_{+}}$, or
\item $a<0$ (focusing) and $0 < \sigma < \frac{2}{d}$.
\end{itemize}
\end{hyp}
Under these hypotheses, the existence of a unique global in-time solution $\psi\in C(\R_t; \Sigma)$ to \eqref{GP} has been proved in \cite{AMS}. 
In particular, the restriction $\sigma < \frac{2}{d}$ in the focusing case ($a<0$) ensures that {\it no finite-time blow-up} can occur.
In addition, the global solution $\psi(t, \cdot)\in \Sigma$ is known to conserve the {\it total mass}, i.e.
\begin{equation}\label{mass}
N(\psi(t, \cdot)) = \int_{\R^d} |\psi(t, x)|^2 \, dx =N(\psi_0), \quad \forall \, t\in \R,
\end{equation}
as well as 
\begin{equation}\label{econ}
E_\Omega(\psi(t, \cdot)) = E_\Omega(\psi_0),  \quad \forall \, t\in \R,
\end{equation}
where $E_\Omega$ denotes the associated {\it Gross-Pitaevskii energy}:
\begin{equation}\label{GPF}
E_\Omega(\psi)= \int_{\R^d} \frac12 |\nabla \psi|^2 + V(x) |\psi|^2 + \frac{a}{\sigma +1} |\psi|^{2\sigma +2} - \overline \psi (\Omega\cdot L) \psi \, dx.
\end{equation}
Note that the last term within $E_\Omega$ is sign indefinite.

In the following, we shall focus on various stability and/or instability properties of solutions $\psi$ to \eqref{GP}: 
Our first task will be to study the orbital stability of nonlinear {\it ground states} associated to \eqref{GP}. These are solutions to \eqref{GP} given by
\[
\psi(t,x) = e^{- i \mu t} \varphi(x),\quad \mu \in \R,
\]
where $\varphi$ is obtained as a constrained minimizer of the energy functional $E_\Omega(\varphi)$. In \cite{IM, RS1, RS2}, the onset of vortex nucleation is linked to a symmetry breaking phenomenon for 
minimizers of $E_\Omega(\varphi)$, which is proved to happen for $|\Omega|$ above a certain critical speed $\Omega_{\rm crit}>0$,  
even in the case of radially symmetric traps $V$ with $\omega_1=\omega_2=\omega_3$ (see Section \ref{sec:stab} for more details). 
In our first main result below, we shall prove that 
under Assumption \ref{hyp1} and for $|\Omega|< \omega$, 
the set of all energy minimizers is indeed orbitally stable under the time-evolution of \eqref{GP}. In turn, this will allow us to conclude 
several new results of orbital stability for a class of rotating solutions to nonlinear Schr\"odinger equations {\it without} the angular momentum term $\propto \Omega$. 

The question of whether the condition $|\Omega|< \omega$ is only needed for the existence of ground states, or also has a nontrivial effect in the solution of the 
time-dependent equation \eqref{GP}, then leads us to our second line of investigation. A theorem based on the Ehrenfest equations associated to \eqref{GP}, shows that
in the case of {\it non-istotropic potentials} $V$, a resonance-type phenomenon can occur for $|\Omega|\ge \omega$. This leads to 
solutions $\psi$ whose $\Sigma$-norm is growing (forward or backward) in time with a rate that can even be exponential, depending on the choice of $\Omega$ and $\omega_j$. 
Physically, this can be interpreted as a manifestation of non-trapped solutions of \eqref{GP} whose mass is pushed towards spatial infinity.\\

The paper is organized as follows: In Section \ref{sec:exist} below we shall prove the existence of nonlinear ground states. Their orbital stability 
(and several further consequences) is proved in Section \ref{sec:stab}. Finally, we turn to the analysis of possible resonances in Section \ref{sec:instab}.

%%%%%%%%%%%%%%%%%%%%%%%%%%%%%%%%%%%%%%%%%%%%%%%%%%%%%%%%%%%%%%%%%%%%%%%%

\section{Existence of ground states}\label{sec:exist}

In this section we shall prove the existence of time-periodic solutions $\psi(t,x) = e^{- i \mu t} \varphi(x)$ to \eqref{GP}, which satisfy the following nonlinear elliptic 
equation
\begin{equation}\label{eq:stat}
\mu \varphi = \Big(-\frac{1}{2}\Delta + V(x) - (\Omega \cdot L)\Big) \varphi + a|\varphi|^{2\sigma}\varphi.
\end{equation}
Note that if $\varphi $ solves this equation, then so does $\varphi e^{i \theta}$ with $\theta \in \R$, i.e., we have symmetry under gauge transformations.

For any given total mass $N>0$, a particular class of solutions $\varphi\in \Sigma$ to \eqref{eq:stat}, called ground states, 
is obtained by considering the following constrained minimization problem:
\begin{equation}\label{min}
e(N,\Omega) := \inf\lbrace {E}_{\Omega}(\varphi) : \varphi \in \Sigma, \ N(\varphi) = N \rbrace,
\end{equation}
where the infimum can be replaced by a minimum whenever the energy functional \eqref{GPF} is bounded from below. 
In this case $e(N, \Omega)>-\infty$ denotes the ground state energy. 
Note that $E_\Omega(\varphi)$ is well-defined for any $\varphi \in \Sigma$, since 
Assumption \ref{hyp1} and Sobolev's imbedding imply $\Sigma \hookrightarrow L^{2\sigma+2}$ provided $\sigma < \frac{2}{(d-2)_{+}}$. 
Moreover, for any $\gamma > 0$ we have
\begin{align}\label{ang-sig}
\left|\langle \psi, (\Omega\cdot L)\psi\right|  \le \|(\Omega \wedge x)\psi\|_{L^{2}}\|\nabla \psi\|_{L^{2}}  \le \frac{1}{2\gamma}|\Omega|^{2}\| x\psi\|^{2}_{L^{2}} + \frac{\gamma}{2}\|\nabla \psi\|^{2}_{L^{2}},
\end{align}
which in itself follows by rewriting $\Omega \cdot L = (\Omega \wedge x) \cdot \nabla$ and employing Young's inequality.

The existence and orbital stability of ground state solutions will be proved by the same method as in \cite{CLi,CLZ}.  
To this end, we shall first show that the energy functional \eqref{GPF} is coercive, provided the angular velocity $|\Omega|$ 
is less than the smallest trapping frequency:
\begin{proposition}\label{lem:Efun}
Let $|\Omega|< \omega$ and Assumption \ref{hyp1} hold. Then for any $\varphi\in \Sigma$ with $\|\varphi\|_{L^{2}}^{2}=N$, there is a $\delta>0$ such that
\begin{equation}\label{wls}
{E}_{\Omega}(\varphi) \ge \delta \| \varphi  \|_{\Sigma}^{2} - C_{N}\ge 0,
\end{equation}
Moreover, $\varphi\mapsto{E}_{\Omega}(\varphi)$ is weakly lower semicontinuous in $\Sigma$, i.e. 
for $\{ \varphi_{k} \}^{\infty}_{k=1} \subset \Sigma$ such that $\varphi_{k} \rightharpoonup\varphi \in \Sigma$, we have
$$
E_\Omega(\varphi) \le \liminf_{k\to \infty} E_\Omega(\varphi_{k}).
$$
\end{proposition}
\begin{proof}
The coercivity follows from \eqref{ang-sig} and the fact that $V(x) \ge \frac{1}{2}\omega^{2}|x|^{2}$ where $\omega>0$ is defined above. Thus one finds, for $0<\gamma <1$:
\begin{equation}\label{eq:energybound}
{E}_{\Omega}(\varphi) 
\ge \frac{1-\gamma}{2} \|\nabla \varphi\|_{L^{2}}^{2}+ \frac{1}{2}\left(\omega^{2} -\frac{|\Omega|^{2}}{\gamma}\right)\|x\varphi\|^{2}_{L^{2}} + \frac{a}{\sigma + 1}\|\varphi\|_{L^{2\sigma +2}}^{2\sigma +2}.
\end{equation} 
In the case $a>0$, we directly obtain
$$
E_\Omega(\varphi) \ge \frac{1-\gamma}{2} \|\nabla \varphi\|_{L^{2}}^{2}+ 
\frac12\left(\omega^{2} -\frac{|\Omega|^{2}}{\gamma}\right)\|x\varphi\|^{2}_{L^{2}} \ge \delta\|\varphi\|_{\Sigma}^{2} - \frac{N}{2}\,,
$$
where we choose $\gamma\in(0,1)$ such that $\frac{|\Omega|^2}{\omega^2}<\gamma<1$, and we set
$$
\delta = \min \left\{ \frac{1-\gamma}{2},\frac12\left(\omega^{2}-\frac{|\Omega|^{2}}{\gamma}\right) \right\}>0\,.
$$
In the case $a<0$, we first note from the Gagliardo-Nirenberg inequality that 
\begin{equation}\label{eq:GN}
\|u\|_{L^{2\sigma +2}}^{2\sigma +2} \le C_{\sigma,d}\|\nabla u\|_{L^{2}}^{d\sigma} \|u\|_{L^{2}}^{2+\sigma(d-2)},
\end{equation}
with the optimal constant $C_{\sigma,d}>0$ obtained in \cite{MW}, i.e., 
$$
C_{\sigma,d} = \frac{\sigma +1}{ \| Q \|_{L^{2}}^{2\sigma}}\,,
$$
where $Q$ satisfies 
$$
\frac{d\sigma}{2} \Delta Q - \Big( 1 + \frac{\sigma(d-2)}{2} \Big)Q + Q^{2\sigma +1} = 0.
$$
Then applying \eqref{eq:GN} to \eqref{eq:energybound} and employing Young's inequality with 
$$
(p,q) = \Big(\frac{2}{d\sigma}, \frac{1}{1-d\sigma/2} \Big)
$$ 
yields the following lower bound for any $\varepsilon>0$:
\begin{equation*}
\begin{aligned}
E_\Omega(\varphi) &\ge \frac{1-\gamma}{2} \|\nabla \varphi\|_{L^{2}}^{2} - \frac{|a|}{\| Q \|^{2\sigma}_{L^{2}}}\|\nabla \varphi\|_{L^{2}}^{d\sigma}\|\varphi\|_{L^{2}}^{2+\sigma(d-2)}  + \frac{1}{2}\left(\omega^{2} - \frac{|\Omega|^{2}}{\gamma}\right)\|x\varphi\|^{2}_{L^{2}}\\
&\ge \Big( \frac{1-\gamma}{2} - \frac{d\sigma|a|\eps^{p}}{2\| Q \|^{2\sigma}_{L^{2}}}\Big)\|\nabla \varphi\|_{L^{2}}^{2} + \frac{1}{2}\left(\omega^{2} -\frac{|\Omega|^{2}}{\gamma}\right)\|x\varphi\|^{2}_{L^{2}}\\ 
&\qquad- \frac{|a|(1-\frac{d\sigma}{2})}{\| Q \|^{2\sigma}_{L^{2}}\eps^{q}}\|\varphi\|_{L^{2}}^{\frac{2+\sigma(d-2)}{1-\frac{d\sigma}{2}}}\,.
\end{aligned}
\end{equation*}
Now recall $p = \frac{2}{d\sigma} > 1$ and choose $\gamma\in(0,1)$, as above, 
such that $\frac{|\Omega|^2}{\omega^2}<\gamma<1$, and then $\eps>0$ such that 
$$
\frac{d\sigma|a|\eps^{p}}{2\| Q \|^{2\sigma}_{L^{2}}} = \frac{1-\gamma}{4}\,.
$$
After recalling that $\|\varphi\|_{L^{2}}^{2} = N$, we find
\begin{align*}
E_\Omega(\varphi) &= \frac{1-\gamma}{4}\|\nabla \varphi\|_{L^{2}}^{2} + 
\frac{1}{2}\left(\omega^{2} -\frac{|\Omega|^{2}}{\gamma}\right)\|x\varphi\|^{2}_{L^{2}} - \frac{|a|(1-\frac{d\sigma}{2})}{\| Q \|^{2\sigma}_{L^{2}}\eps^{q}}N^{\frac{2+\sigma(d-2)}{2-d\sigma}} \\
&\ge \tilde{\delta}\|\varphi\|_{\Sigma}^{2} - C(a,d,\sigma,N,\| Q \|^{2\sigma}_{L^{2}})
\end{align*}
where 
$$
\tilde \delta = \min \left\{ \frac{1-\gamma}{4},\frac12\left(\omega^{2}-\frac{|\Omega|^{2}}{\gamma}\right) \right\}\,.
$$
Moreover, since the $\Sigma$-norm is weakly lower semicontinuous, the estimate \eqref{wls} directly implies the same 
holds for $E_\Omega$, since its quadratic part together with a multiple of the $L^2$-norm forms a norm on $\Sigma$ equivalent to the usual one.
\end{proof}

To proceed further, we recall the following compactness result (see, e.g., \cite{HHMS, ZGJ}). 

\begin{lemma}\label{lem:cpt}
For $2 \le q <  \frac{2d}{(d - 2)_{+}}$, the embedding
$
\Sigma \hookrightarrow L^{q}
$
is compact.
\end{lemma}

Using this we can prove existence of a (constrained) minimizer.

\begin{proposition}\label{prop:exist1}
Let $|\Omega|< \omega$ and Assumption \ref{hyp1} hold. Then for a given $N>0$, there exists a $\varphi_{\infty} \in \Sigma$ such that $\|\varphi_\infty\|_{L^2}^2=N$ and
$$
{E}_{\Omega}(\varphi_{\infty})=\min_{\varphi \in \Sigma}{E}_{\Omega}(\varphi)=e(\Omega,N).
$$
In addition, $\varphi_\infty$ is a weak solution to \eqref{eq:stat} with $\mu \in \R$ a Lagrange multiplier 
associated to the mass constraint.
\end{proposition}
\begin{proof}
Choose a minimizing sequence $\{ \varphi_{k} \}^{\infty}_{k=1} \subset \Sigma$ such that $\|\varphi_{k}\|_{L^{2}}^{2}=N$.
First we show $\{ \varphi_{k} \}^{\infty}_{k=1}$ is a  bounded sequence in $\Sigma$.
From Proposition \ref{lem:Efun} we know that $0<E_{\Omega}(\varphi_k)<\infty$
and the coercivity implies that any minimizing sequence $\{ \varphi_{k} \}^{\infty}_{k=1}$ is a bounded sequence in $\Sigma$.
By Banach-Alaoglu, there exists a weakly convergent subsequence $\{ \varphi_{k_{j}} \}^{\infty}_{j=1} \subset \{ \varphi_{k} \}^{\infty}_{k=1}  $ such that 
$$
\varphi_{k_{j}} \rightharpoonup \varphi_{\infty}  \quad \mbox{as} \ j \to \infty,
$$
for some $\varphi_{\infty} \in \Sigma$. The compact embedding of Lemma \ref{lem:cpt} implies that  $\varphi_{k_{j}} \to \varphi_{\infty}$ 
strongly (and hence in norm) in $L^{2}$ and in $L^{2\sigma +2}$, provided $\sigma < \frac{2}{(d-2)_{+}}$. In particular
\begin{align}
\|\varphi_{\infty}\|_{L^{2}}^{2} = \lim_{j\to \infty}\|\varphi_{k_{j}}\|_{L^{2}}^{2} = N.
\end{align}
By the lower semicontinuity of the functional $E_\Omega$ we have
$$
E_{N} := \inf_{\varphi \in \Sigma, \| \varphi\|_2^2=N} E_\Omega(\varphi) \le E_\Omega(\varphi_{\infty}) \le \lim_{j\to \infty} \inf E_\Omega(\varphi_{k_{j}}) = E_{N}.
$$
Furthermore, since $e(N, \Omega)\equiv E_\Omega(\varphi_{\infty})=\lim_{j\to \infty}E_\Omega(\varphi_{k_{j}})$, 
we see that $\|\varphi_{k_{j}}\|_{\Sigma} \to \|\varphi_{\infty}\|_{\Sigma}$, as $j \to \infty$. 
Together with the weak convergence of the minimizing sequence this implies strong convergence to some $\varphi_{\infty} \in \Sigma$.

It is then straightforward to compute the first variation $\langle\frac{\delta E_\Omega}{\delta \varphi}, \chi\rangle=0$ to see that a minimizer $\varphi_\infty \in \Sigma$ indeed 
solves \eqref{eq:stat} in the weak sense, i.e.
\[
\mu \int_{\R^d} \overline \varphi_\infty \chi \, dx = \frac{1}{2}\int_{\R^d} \nabla  \overline \varphi_\infty  \cdot \nabla  \chi + V(x)  \overline\varphi_\infty   \chi 
-  \overline\varphi_\infty  (\Omega \cdot L)  \chi+ a|\varphi_\infty |^{2\sigma} \overline\varphi_\infty  \chi\, dx,
\] 
for all $\chi \in \Sigma$.
\end{proof}

\begin{remark}
It is straightforward to generalize all of the results in this section to GP equations with general confinement potentials $V(x)\to +\infty$, as $|x|\to \infty$, provided 
an appropriate energy space $\Sigma$ is chosen.
\end{remark}

%%%%%%%%%%%%%%%%%%%%%%%%%%%%%%%%%%%%%%%%%%%%%%%%%%%%%%%%%%%%%%%%%%%%%%%%

\section{Orbital stability}\label{sec:stab}

The set of all ground states with a given mass $N$ will be denoted by
\begin{equation}\label{Gset}
\mathcal G_\Omega = \Big\lbrace \varphi \in \Sigma  : {E}_{\Omega}(\varphi) = e(N,\Omega)  \ {\rm and} \ N(\varphi) = N  \Big\rbrace \ne \emptyset.
\end{equation} 
Recall that, by gauge symmetry, $\varphi \in \mathcal G_\Omega$ if and only if $e^{i \theta}\varphi\in \mathcal G_\Omega$, for some $\theta \in \R$.
In the case without rotation, i.e., $\Omega \equiv 0$, and for radially symmetric potentials $V$ with $\omega_1=\omega_2=\omega_3$, 
one can show that the energy minimizer is indeed radially symmetric and positive on all of $\R^d$, see \cite{HHMS, HiO} and the references therein. 
In other words, in this case
\begin{equation}\label{Gsimple}
\mathcal G_0 = \{ ue^{i \theta}, \ u\equiv u(|x|)>0, \, \theta \in \R \}.
\end{equation}
Moreover, since the action of $\Omega\cdot L$ vanishes on radially symmetric functions, any radially symmetric $\varphi\in \mathcal G_\Omega$ is also 
in $\mathcal G_0$, and hence 
of the form above. However, the symmetry breaking results in \cite{RS1, RS2} imply that for $|\Omega|\not =0$, a minimizer $\varphi_\infty\in \mathcal G_\Omega$ is in general 
{\it not radially symmetric}. More precisely, it is proved 
in there that for $|\Omega|> \Omega_{\rm crit}>0$  
no eigenfunction of the angular momentum operator $ L$ can be a minimizer (and a radial function $u$ is an eigenfunction with zero eigenvalue), 
even if the GP functional is invariant under rotations around the $\Omega$-axis. 
This implies that $\varphi_\infty$ in the case {\it with} rotation cannot be unique (up to gauge transforms), since by rotating a minimizer 
one obtains another minimizer. In this context, an estimate for the critical rotation speed $\Omega_{\rm crit}$ in $d=2$ can be found in \cite{IM}. 
In summary, these results show that $\mathcal G_\Omega$, in general, will be a more complicated set than $\mathcal G_0$. Moreover, $\mathcal G_\Omega$ should 
also be distinguished from the set of rotationally symmetric vortex solutions studied in, e.g., \cite{GAP}.

Our first main result is as follows:

\begin{theorem}[Orbital stability of ground states]\label{thm:OG}
Let $|\Omega|< \omega$ and Assumption \ref{hyp1} hold. Then the set of ground states $\mathcal G_\Omega\ne \emptyset$ is orbitally stable in $\Sigma$.  
That is, for all $\eps>0$ there exists $\delta=\delta(\eps)>0$, such that if $\psi_0\in \Sigma$ satisfies
$$
\inf_{\varphi \in \mathcal G_\Omega}\| \psi_{0} - \varphi \|_{\Sigma} < \delta,
$$
then the solution $\psi\in C(\R_t, \Sigma)$ 
to \eqref{GP} with $\psi (0,x) = \psi_0 \in \Sigma$ satisfies 
$$
\sup_{t \in \R}\inf_{\varphi \in \mathcal G_\Omega}\| \psi(t,\cdot)- \varphi \|_{\Sigma}< \eps.
$$
\end{theorem}

This theorem generalizes earlier results on the orbital stability of standing waves in nonlinear Schr\"odinger equations with (unbounded) potential (see, e.g., \cite{CLi, RF, RO3, HHMS, ZGJ, Z1} and the references therein) to the case with harmonic potential and additional rotation. Note that for $\Omega =0$, the 
simple structure of $\mathcal G_0$, given in \eqref{Gsimple}, allows one to rephrase the infimum over $\mathcal G_0$ as an infimum over $\theta \in \R$.
Also note Theorem \ref{thm:OG} holds for defocusing and focusing nonlinearities 
satisfying Assumption \ref{hyp1} (see also Remark \ref{rem1} below). In this context, we also mention the papers \cite{RF, RO2}, in which the authors study various  
instability properties of standing wave solutions to focusing nonlinear Schr\"odinger equations with potentials. 

\begin{proof}
By way of contradiction, assume that the set of ground states $\mathcal G_\Omega\ne \emptyset$ is unstable. 
Then there exist $\eps_{0}>0$, $\varphi_0 \in \mathcal G_\Omega$, 
a sequence of initial data $\lbrace \psi_{0}^{k} \rbrace_{k \in \N} \subset \Sigma$ satisfying
\begin{align*}
&\| \psi_{0}^{k} - \varphi_0 \|_{\Sigma} \to 0 \quad {\rm as} \quad k \to \infty,
\end{align*}
and a sequence of times $\lbrace t_{k} \rbrace_{k \in \N} \subset \R$, such that
$$
\inf_{\varphi \in \mathcal G_\Omega}\| \psi^{k}(t_{k},\cdot)-\varphi\|_{\Sigma} > \eps_0.
$$
Here $\psi^{k}(t,x) \in C(\R,\Sigma)$ is the unique global solution to \eqref{GP} with initial data $\psi_{0}^{k}$.  
For simplicity set $u_{k}(x) := \psi^{k}(t_{k},x)$. From mass conservation \eqref{mass} we have, as $k\to \infty$:
\begin{equation*}
\|u_{k}\|_{L^{2}}^{2} \equiv \|\psi^{k}(t_{k},\cdot)\|_{L^{2}}^{2} = \|\psi_{0}^{k}\|_{L^{2}}^{2} \xrightarrow{k \to \infty} \|\varphi_0 \|_{L^{2}}^{2} = N.
\end{equation*}
Moreover, by energy conservation \eqref{econ} it also follows that
\begin{equation*}
{E}_{\Omega}(u_{k}) \equiv {E}_{\Omega}(\psi^{k}(t_{k},\cdot)) = {E}_{\Omega}(\psi_{0}^{k}) \xrightarrow{k \to \infty}  {E}_{\Omega}(\varphi_0) = e(N,\Omega).
\end{equation*}
Consequently, the continuity in time implies that $u_k$ is a minimizing sequence in $\Sigma$. 
By the proof of Proposition \ref{prop:exist1}, there exists a subsequence such that $u_{k_{j}} \to \varphi_{\infty} \in \Sigma$ strongly, as $j\to \infty$. Thus 
$$
\inf_{\varphi \in \mathcal G_\Omega}\| \psi^{k_{j}}(t_{k_{j}},\cdot)- \varphi\|_{\Sigma} \le \|u_{k_{j}} - \varphi_{\infty}\|_{\Sigma}  \xrightarrow{j \to \infty}  0,
$$
which contradicts our assumption.
\end{proof}

\begin{remark}\label{rem1}
It is possible to generalize this result to the case of an attractive ($a<0$) mass-critical nonlinearity $\sigma= \frac{2}{d}$, under the assumption that $N< \| Q\|_{L^2}^2$, see, e.g., 
\cite{Z1, Z2} for analogous results in the case without rotation. We shall not go into further details here, but note that the associated question of a blow-up profile as $N\to \| Q\|_{L^2}^2$ in 
the case with rotation has recently been studied in \cite{LNR}. 
\end{remark}

Theorem \ref{thm:OG} has the following interesting consequence: Recall that $\Omega\cdot L$ is the generator of rotations around the 
$\Omega $-axis, in the sense that
\[
e^{-it \Omega \cdot L } u(x)= u\left(e^{t\Theta }x \right), \quad \forall \, u \in L^2(\R^d),
\]
where $\Theta$ is the skew symmetric matrix given by
\[
\Theta=
  \begin{pmatrix}
    0 & |\Omega| \\
    -|\Omega| & 0 
  \end{pmatrix}\ \text{for $d=2$, and } 
  \Theta=
  \begin{pmatrix}
    0 & \Omega_{3} & -\Omega_{2} \\
    -\Omega_{3} & 0 & \Omega_{1} \\
    \Omega_{2} & -\Omega_{1} & 0
  \end{pmatrix} \ \text{for $d=3$.}
\]
Clearly, this is a unitary operator on both $L^2(\R^d)$ and $\Sigma$. It was shown in \cite{AMS} that if $\psi(t,x)$ solves \eqref{GP}, i.e., the GP equation with rotation, then 
\begin{equation}\label{tchange}
\Psi(t,x): = \big(e^{-it \Omega \cdot L } \psi(t,\cdot)\big)(x),
\end{equation}
solves the following nonlinear Schr\"odinger equation with time-dependent potential:
\begin{equation}\label{NLS}
	i \partial_{t}\Psi = - \frac{1}{2}\Delta \Psi+ W_\Omega(t,x)\Psi+ a|\Psi|^{2\sigma}\Psi, \quad \Psi_{\vert{t=0}} = \psi_{0}(x)\,.
\end{equation}
Here, the new potential $W_\Omega$ is given by
\[
W_\Omega(t,x) := e^{-it \Omega \cdot L }V (x)\equiv V\left(e^{t\Theta }x \right).
\]
The global existence result for \eqref{GP} then directly translates to the existence of a unique global solution $\Psi\in C(\R_t; \Sigma)$ to \eqref{NLS} (see also \cite{Car} for related results). 
Moreover, we have that \eqref{NLS} conserves the total mass, i.e., $N(\Psi(t,\cdot)) = N(\psi_0)$ for all $t\in \R$. 
The associated energy, however, is {\it no longer conserved} unless $V(x)$ is rotationally 
or at least axisymmetric w.r.t. $\Omega$, cf. \cite{AMS} for more details.

\begin{corollary}\label{cor:asym}
Under the same assumptions as in Theorem \ref{thm:OG} it holds: 
For all $\eps>0$ there exists $\delta=\delta(\eps)>0$ such that if $\psi_{0} \in \Sigma$ satisfies 
$$
\inf_{\varphi \in \mathcal G_\Omega}\| \psi_{0} - \varphi \|_{\Sigma} < \delta,
$$
then the solution $\Psi\in C(\R_t, \Sigma)$ to \eqref{NLS} with $\psi (0,x) = \psi_0 \in \Sigma$ satisfies 
$$
\sup_{t \in \R}\inf_{\varphi \in \mathcal G_\Omega}\| \Psi(t,\cdot)- e^{-it \Omega \cdot L } \varphi(\cdot)\|_{\Sigma}< \eps.
$$
\end{corollary}
In other words, we have orbital stability of the set $e^{-it \Omega \cdot L }\mathcal G_\Omega$ under the dynamics of \eqref{NLS}. 
To the best of our knowledge, this is the only orbital stability result for nonlinear Schr\"odinger equations with a time-dependent potential available to date. 

In the particular situation where $V$ is {\it rotationally symmetric}, i.e., $V(x) = \frac{1}{2}\omega^2|x|^2$, one finds
\[
W_\Omega(t,x)=V(x), \quad \text{for any $\Omega \in \R^d$}, 
\]
yielding the usual Gross-Pitaevskii equation for (harmonically) trapped Bose gases
\begin{equation}\label{NLSV}
	i \partial_{t}\Psi = - \frac{1}{2}\Delta \Psi+ \frac{1}{2}\omega^2|x|^2+ a|\Psi|^{2\sigma}\Psi, \quad \Psi_{\vert{t=0}} = \psi_{0}(x)\,,
\end{equation}
In contrast to \eqref{NLS}, this equation does conserve the associated 
Gross-Pitaveskii energy, $E_0(\Psi(t,\cdot))=E_0(\psi_0)$, for all $t\in \R$. The orbital stability result proved above then has the following 
consequence:

\begin{corollary}\label{cor:sym}
Let Assumption \ref{hyp1} hold and $V$ be rotationally symmetric. Then
\begin{equation*}
\mathcal O = \cup_{(\Omega\in \R^d; |\Omega|<\omega)} \left( e^{-it \Omega \cdot L }\mathcal G_\Omega\right),
\end{equation*}
is an orbitally stable set of solutions to \eqref{NLSV}.
\end{corollary}

The usual orbital stability result for ground states associated to \eqref{NLSV} applies to $\mathcal G_0$, see, e.g., \cite{CLi}. 
Note that if, for some $\Omega$, all minimizers $\varphi\in \mathcal G_\Omega$ are rotationally symmetric, then 
$e^{-it \Omega \cdot L }\mathcal G_\Omega = \mathcal G_\Omega=\mathcal G_0$. 
However, the results of \cite{IM, RS1, RS2} show that, in general, $\varphi\in \mathcal G_\Omega$ is {\it not} rotationally symmetric, 
in which case $ e^{-it \Omega \cdot L }\mathcal G_\Omega$, does {\it not} contain stationary solutions to \eqref{NLSV} given by $\Psi(t,x) = \Phi(x)e^{i \mu t}$. 
Again, to the best of our knowledge, this is the only orbital stability result for \eqref{NLSV} based on non-stationary solutions.

%%%%%%%%%%%%%%%%%%%%%%%%%%%%%%%%%%%%%%%%%%%%%%%%%%%%%%%%%%%%%%%%%%%%%%%%

\section{A resonance-type phenomenon in non-isotropic potentials} \label{sec:instab}

All the preceding results are obtained under the condition $|\Omega|<\omega$, which is necessary for the existence of nonlinear ground states. 
However, one may wonder (in particular in view of Corollary \ref{cor:sym}) if there are any qualitative changes to the time-dependent solution of 
\eqref{GP} for $|\Omega|\ge \omega$. At least in the case of {\it non-isotropic potentials} $V(x)$, we will see below that this is indeed the case.

To this end, we denote for $\psi(t,\cdot)\in \Sigma$, the quantum mechanical mean position and momentum by
\[
X (t) := \int_{\R^d} x|\psi (t,x)|^2\, dx,  \quad P (t) := -i \int_{\R^d} \overline \psi (t,x) \nabla \psi(t,x) \, dx.
\]

\begin{lemma}
Let $\psi\in C(\R_t; \Sigma)$ be a solution to \eqref{GP}, then, for all $t\in \R$:
\begin{equation}\label{ode}
\begin{aligned}
& X (t) = X(0) +\int_0^t P (s) -  \Omega \wedge X  (s)\, ds   \\
& P(t)  =P(0) -\int_0^t \nabla V(X(s)) + \Omega \wedge P(s) \, ds.
\end{aligned}
\end{equation}
\end{lemma}

This system can be regarded as a generalization of the results in \cite[Section 6]{Oh}, obtained for $\Omega=0$. Note that the nonlinearity does not enter in \eqref{ode}.

\begin{proof} We shall assume that $\psi$ is sufficiently smooth (and decaying) such that all of our computations below are rigorous. 
A classical density argument, combined with the continuous dependence of $\psi$ on its initial data, then allows us to extend the result 
to solutions $\psi\in C(\R_t; \Sigma)$. 

We start by calculating the time derivative of $X$:
\begin{align*}
\dot X &= 2 \re \langle \partial_{t}\psi, x \psi   \rangle = 2 \re \langle i( \tfrac{1}{2}\Delta \psi - V(x)\psi - a|\psi|^{2\sigma}\psi + (\Omega \cdot L) \psi ), x \psi   \rangle \\
&= \re \langle i \Delta \psi , x \psi   \rangle  + 2\re \langle i (\Omega \cdot L) \psi , x \psi   \rangle   + 2 \im \underbrace{\langle  V(x)\psi + a|\psi|^{2\sigma}\psi, x \psi   \rangle}_{\in \R} \\
&\equiv J_{1} + J_{2}.
\end{align*}
An integration by parts then implies
\begin{align*}
J_{1} &=\re \langle -i \nabla \psi , \nabla( x \psi ) \rangle = \im {\langle \nabla \psi , x \nabla\psi  \rangle} + \re \langle -i \nabla \psi , \psi\nabla x  \rangle =P.
\end{align*}
The term $J_2$ can be rewritten using $(\Omega \cdot L) = -i(\Omega \wedge x) \cdot \nabla$ and integration by parts 
\begin{align*}
J_{2} &=  2\re \langle (\Omega \wedge x) \cdot \nabla\psi , x \psi   \rangle = 2\re\sum_{\ell, j =1}^{d}\langle  \partial_{x_{j}}\psi , (\Omega \wedge x)_{j} x_{\ell} \psi   \rangle e_{\ell} \\
&= -2\sum_{j=1}^{d}\langle \psi, (\Omega \wedge x)_{j} \psi\rangle e_{j} -2\re\langle  x \psi , (\Omega \wedge x) \cdot \nabla \psi  \rangle  =-2\langle \psi, (\Omega \wedge x) \psi\rangle -J_{2},
\end{align*}
which implies that 
\begin{equation*}\label{j2}
J_{2} = -\langle \psi, (\Omega \wedge x) \psi\rangle = - \Omega \wedge X .
\end{equation*}
In summary this yields the following equation of motion for $X$:
\begin{equation}\label{Xt}
\dot X = P -  \Omega \wedge X ,
\end{equation}
which is the time-differentiated version of the first equation in \eqref{ode}.

Next, we calculate the time-derivative of $P$ as:
\begin{align*}
\dot P &= 2\re\langle i\partial_{t}\psi, \nabla \psi  \rangle = 2\re\langle  V(x)\psi + a|\psi|^{2\sigma}\psi - \tfrac{1}{2}\Delta \psi -  (\Omega \cdot L)\psi , \nabla \psi  \rangle \\
&\equiv I_{1} + I_{2} + I_{3} +I_{4}.
\end{align*}
For the first term, a straightforward integration by parts yields
\begin{align*}
I_{1} =  -2\re\langle   \nabla \big( V \psi \big),  \psi \rangle = -2\int_{\R^{d}}\nabla V(x) |\psi(t,x)|^{2}\;dx - I_{1},
\end{align*}
which implies 
\begin{equation*}\label{i1}
I_{1} = -\int_{\R^{d}}\nabla V(x) |\psi(t,x)|^{2}\;dx = - \nabla V(X),
\end{equation*}
since $\nabla V(x) = \sum_{j=1}^d \omega_{j}^{2} x_j$. Furthermore, $I_2$ vanishes, since
\begin{align*}
I_{2} =  \frac{a}{\sigma+1} \int_{\R^{d}} \nabla \big(|\psi|^{2(\sigma +1)} \big) \;dx  = 0,
\end{align*}
and one also finds
$
I_{3} = -\re\langle  \Delta \psi, \nabla \psi \rangle= 0.
$
Finally, we compute, using standard vector identities 
\begin{align*}
I_{4} = -2\re\langle  (\Omega \cdot L)\psi, \nabla \psi \rangle = -2\Omega \wedge P  - I_{4},
\end{align*}
which implies that 
\begin{align}\label{Pt}
\dot P  =-  \nabla V(X) - \Omega \wedge P, 
\end{align}
i.e., the differential version of the second line in \eqref{ode}.
\end{proof}

Given that \eqref{ode} constitutes a closed system for $X$ and $P$, one can study its solution independently of \eqref{GP}. As a first step, we 
have the following global existence result.
\begin{lemma}
For any $(X_0, P_0)\in \R^{2d}$, the system \eqref{ode} admits a unique global in-time solution $(X, P)\in C^\infty(\R_t; \R^{2d})$ with $(X(0), P(0))=(X_0, P_0)$.
\end{lemma}

\begin{proof}
Denote $\Xi= (X, P)^\top$, then \eqref{Xt}, \eqref{Pt} are equal to
\begin{equation}\label{mode}
   \dot{\Xi} = M_d\Xi,\quad \Xi(0) = \Xi_0,
\end{equation}
where $\Xi_0 = (X_0, P_0)^\top$, and
\[
M_2=
  \begin{pmatrix}
   0 & |\Omega| & 1 & 0 \\
   -|\Omega| & 0 & 0 & 1 \\
   -\omega_{1}^{2} & 0 & 0 & |\Omega|  \\
   0 & -\omega_{2}^{2} & -|\Omega| & 0 \\
   \end{pmatrix}\quad \text{for $d=2$},
\]
and
\[
M_3=
   \begin{pmatrix}
   0 & \Omega_{3} & -\Omega_{2} & 1 & 0 & 0\\
   -\Omega_{3} & 0 & \Omega_{1} & 0 & 1 & 0\\
   \Omega_{2} & -\Omega_{1} & 0 & 0 & 0 & 1 \\
   -\omega^{2}_{1} & 0 & 0 & 0 & \Omega_{3} & -\Omega_{2}\\
 0 & -\omega^{2}_{2} & 0 & -\Omega_{3} & 0 & \Omega_{1}\\
   0 & 0 & -\omega^{2}_{3} & \Omega_{2} & -\Omega_{1} & 0
  \end{pmatrix}\quad \text{for $d=3$}.
\]
Equation \eqref{mode} is a linear matrix-valued ordinary differential equation with constant coefficients. Thus, \eqref{mode}, and equivalently \eqref{ode}, admits 
a unique smooth solution given by:
\[
\Xi(t) = e^{tM_d} \Xi_0, \quad \text{for all $t\in \R$.}
\]
\end{proof}

To simplify the following discussion, we shall assume that $\Omega\in \R^3$ is aligned with one of the coordinate axes, say, $\Omega = (0,0, |\Omega|)^\top$.
In this way, \eqref{op2d} automatically holds and thus the two-dimensional situation is included in what follows. 

\begin{proposition} \label{prop:ode}
Let $\Omega = (0,0, |\Omega|)^\top$. Assume that 
\begin{equation}\label{cond}
\omega_1 \not = \omega_2 \ \text{and} \ \min\{\omega_1 , \omega_2\} \le | \Omega | \le \max\{\omega_1 , \omega_2\}.
\end{equation}
Then for all $(X_0,P_0)\in\R^{2d}\setminus \mathcal H$, where $\mathcal H=\mathcal H(\omega_1,\dots,\omega_d,\Omega)$ 
is a linear subspace of $\R^{2d}$, it holds
\[
\lim_{t \to + \infty} |X (t)| = \lim_{t \to + \infty} |P(t)| = +\infty, \quad or \quad \lim_{t \to - \infty} |X (t)| = \lim_{t \to - \infty} |P(t)| =+\infty.
\]
Moreover, if both inequalities in \eqref{cond} are strict, this growth is exponentially fast and  $\dim \mathcal H = 2(d-1)$. 
If, however $|\Omega| \in \{ \omega_1, \omega_2 \}$, then the growth is only linear in time and $\dim \mathcal H = 2d-1$.
\end{proposition}

\begin{proof} Observe that for $\Omega = (0,0, |\Omega|)^\top$, the matrix $M_3$ decomposes as a direct sum 
of $M_2$ and the $2\times 2$ matrix
\[
A=
\begin{pmatrix}
   0 &  1  \\
   -\omega_3^2 & 0 \\
\end{pmatrix}.
\]
Thus the characteristic polynomial of $M_3$ is
\[
\det (\lambda - M_3) = \det (\lambda - M_2)\cdot \det (\lambda - A)=\det (\lambda - M_2)\cdot (\lambda^2 + \omega_3^2).
\]
Note that $\lambda^2 + \omega_3^2$ has purely imaginary roots, leading to bounded oscillations in the solution of \eqref{ode}. Thus, for 
both $d=2$ and $d=3$ the characteristic polynomial of $M_2$ is the only possible source of growth in the solution. One finds that
$$
\det (\lambda - M_2)=\lambda^{4} + b\lambda^2+c
$$
with
$$
b=2|\Omega|^{2} + \omega_{1}^{2} + \omega_{2}^{2}
\quad\text{and}\quad
c=\big( |\Omega|^{2} - \omega_{1}^{2}\big)\big(|\Omega|^{2} -\omega_{2}^{2}\big).
$$
As a quadratic polynomial in $\lambda^2$, it has discriminant
\[
D=\big(\omega_1^2-\omega_2^2\big)^2+8|\Omega|^2\big(\omega_1^2+\omega_2^2\big)> 0,
\]
and thus $\lambda^2\in\R$. This implies that a necessary condition for the fact that at least one of the two limits 
\[
\lim_{t\to\pm\infty} |\Xi(t)|=+\infty,
\]
is that $\lambda^2\ge0$. This growth occurs on $\R^{2d}\setminus \mathcal H$, 
where $\mathcal H$ is the orthogonal complement of the eigenspace corresponding to the real eigenvalue(s) $\lambda$.

Computing the roots, we find that since $b>0$, the root
$$
\lambda^2=\frac{-b-\sqrt{b^2-4c}}{2}<0.
$$
In addition, the other root satisfies
$$
\lambda^2=\frac{-b+\sqrt{b^2-4c}}{2}\ge 0,\ \text {if and only if $c\le 0$.}
$$
The latter is equivalent to $\min\{\omega_1 , \omega_2\} \le | \Omega | \le \max\{\omega_1 , \omega_2\}.$

Now if $c<0$ then $\lambda^2>0$. Hence, the system has a positive and a negative simple eigenvalue, implying exponential growth for $t\to \pm \infty$ 
and co-dimension of $\mathcal H$ equal to $2$. The fact that both $X$ and $P$ grow individually can be seen from computing the 
eigenvector $V=(v_1,v_2, v_3, v_4)^\top$ associated to $\lambda$. 
This can be done using the block structure of $M_2$ to derive a new eigenvalue equation for $(v_1, v_2)^\top$, given by
\[
\begin{pmatrix}
   |\Omega|^2-\omega_1^2 &  -2\lambda |\Omega|  \\
   2\lambda |\Omega| & |\Omega|^2-\omega_2^2 \\
\end{pmatrix}
\begin{pmatrix}
v_1\\ v_2
\end{pmatrix} = \lambda^2 \begin{pmatrix}
v_1\\ v_2
\end{pmatrix}.
\]
In addition, one finds that 
\[
\begin{pmatrix}
v_3\\ v_4
\end{pmatrix}
= 
\begin{pmatrix}
\lambda & |\Omega| \\ 
- |\Omega| & \lambda
\end{pmatrix}
\begin{pmatrix}
v_1\\ v_2
\end{pmatrix}.
\]
This yields the expression for $V$ after which a straightforward but somewhat tedious analysis leads to the desired conclusion.

When $c=0$ then $\lambda=0$ is a double eigenvalue, in which case one needs to study the dimension $d_0\in \N$ of the 
associated eigenspace. A straightforward computation shows that if $\omega_1 = \omega_2$ (the axisymmetric case), then $d_0 = 2$ is maximal and hence 
the solution does not grow in $t$. By contrast if $\omega_1 \not = \omega_2$, then $d_0 = 1$, and there exists a linearly independent solution 
$\propto t$, stemming from the eigenvector $V=(1,0,0, -|\Omega|)^\top$. 
\end{proof}

\begin{remark} In the case without rotation, i.e. $|\Omega| = 0$, one finds
$$
\lambda^{2} = - \frac{\omega_{1}^{2} + \omega_{2}^{2}}{2} \pm \big|\frac{\omega_{1}^{2}-\omega_{2}^{2}}{2}\big|,
$$
which implies $\lambda = \pm i\omega_{1},\pm i\omega_{2}$, and thus a purely oscillatory solution.  
\end{remark}

We are now in position to prove the second main result of this work.

\begin{theorem}[Resonance in non-isotropic potentials]\label{thm:inst}
Let Assumption \ref{hyp1} hold and $\Omega = (0,0, |\Omega|)^\top$. If condition \eqref{cond} holds and if $\psi_0\in \Sigma$ is such that 
the associated averages $(X_0, P_0)\not \in \mathcal H$, then the solution $\psi\in C(\R_t; \Sigma)$ satisfies
\[
\lim_{t\to +\infty} \| \psi(t, \cdot) \|_{\Sigma} = \infty, \ \text{or} \, \lim_{t\to -\infty} \| \psi(t, \cdot) \|_{\Sigma} = \infty.
\]
\end{theorem}

\begin{proof}
Recall that both \eqref{GP} and \eqref{ode} have unique solutions. Thus, if $\psi(t, \cdot)$ solves \eqref{GP} with initial data 
$\psi_0\in \Sigma $ and if $X_0=\langle \psi_0, x \psi_0\rangle$ and 
$P_0=-i \langle \psi_0, \nabla \psi_0\rangle$ are the initial data to \eqref{ode}, then 
\[
X(t)=\langle \psi(t, \cdot), x \psi(t, \cdot)\rangle, \quad P(t)=-i\langle \psi(t, \cdot), \nabla \psi(t, \cdot)\rangle, \quad \forall \, t\in \R.
\]
By Cauchy-Schwarz
\[
|X| \le \| \psi \|_{L^2} \| x \psi \|_{L^2}, \quad |P| \le \| \psi \|_{L^2} \| \nabla \psi \|_{L^2},
\]
which together with the results of Proposition \ref{prop:ode} and the mass conservation property \eqref{mass} implies the assertion of the theorem.
\end{proof}

\begin{remark} The fact that there are nontrivial $\psi_0 \in \Sigma$ for which the associated $(X_0, P_0)\not \in \mathcal H$, can be easily seen by considering
initial data of the form:
\[
\psi_0(x) = e^{i p_0\cdot x} e^{-(x-x_0)^2/2}, \quad x_0, p_0 \in \R^d.
\]
In this case, $X_0=\pi^{d/2} x_0$ and $P_0 = \pi^{d/2} p_0$ and thus one obtains a growing $\Sigma$-norm of the solution $\psi$ provided $(x_0, p_0)\not \in \mathcal H$.
\end{remark}

Indeed, the proof of Proposition \ref{prop:ode} shows that if condition \eqref{cond} holds, there are solutions to \eqref{GP} for which
\[
\| \nabla \psi (t, \cdot) \|_{L^2}, \| x \psi (t, \cdot) \|_{L^2} \to \infty,
\] 
if $t\to +\infty$, or $t\to - \infty$. 
In other words, these solutions develop frequencies which are larger than those controlled by the $\Sigma$-norm and, in addition, their mass is transferred to 
infinity, resulting in a weaker decay of $\psi$. This is in sharp contrast to the case $\omega_1 = \omega_2=\omega_3$, where \eqref{GP} 
is equivalent, up to the time-dependent change of variables \eqref{tchange}, to the classical NLS with harmonic trapping \eqref{NLSV}. The latter 
conserves the energy $E_0(\Psi(t, \cdot))=E_0(\psi_0)$, 
which in the defocusing case $a>0$ directly yields the uniform bound 
\[
\| \Psi (t, \cdot) \|_\Sigma = \| \psi(t,\cdot) \|_{\Sigma} \le E_0(\psi_0), \quad \forall \, t\in \R.
\]

\begin{remark}
The growth of (higher order) Sobolev-norms of solutions to nonlinear Schr\"odinger equations with time-dependent, quadratic potentials 
was also studied in \cite{Car}. 
One can check that \eqref{NLS} (obtained from \eqref{GP}, via the change of variables) falls into the class of models 
for which exponentially growing upper bounds were established in \cite{Car}. Theorem \ref{thm:inst} shows that, in general, such 
exponential growth indeed occurs, and that this is true even for linear Schr\"odinger equations. There exponential growth 
naturally occurs in the case of (even only partially) repulsive harmonic potentials. 
We finally mention that very recently a somewhat 
similar instability phenomenon for linear Schr\"odinger equations with quadratic time-dependent Hamiltonian 
has been established in \cite{BGMR}.
\end{remark}

It is very likely that additional (in-)stability phenomena appear for general $\Omega\in \R^3$, not necessarily aligned to one of the axis. However, the 
calculations of the roots of the associated degree $6$ characteristic polynomial become extremely involved, see also \cite{BS}. Since our main goal was to 
establish an instability result for $\psi$ we do not investigate the general case in full detail. 

%%%%%%%%%%%%%%%%%%%%%%%%%%%%%%%%%%%%%%%%%%%%%%%%%%%%%%%%%%%%%%%%%%%%%%%%
%%%%%%%%%%%%%%%%%%%%%%%%%%%%%%%%%%%%%%%%%%%%%%%%%%%%%%%%%%%%%%%%%%%%%%%%

%\section{Conflict of interest statement}
%On behalf of all authors, the corresponding author states that there is no conflict of interest. 
%

\bibliographystyle{amsplain}

\end{document}